\documentclass[draft]{amsart}
\usepackage{amssymb, enumerate, xspace, graphicx}
\usepackage[all]{xy}
\usepackage[english]{babel}
\usepackage[T1]{fontenc}
\usepackage[latin1]{inputenc}
\SelectTips{cm}{}

\numberwithin{equation}{section}

1

\numberwithin{subsection}{section}

\allowdisplaybreaks[1]

\newtheorem*{namedtheorem}{\theoremname}
\newcommand{\theoremname}{testing}

\newtheorem{theorem}{Theorem}[section]

\newtheorem{proposition}[theorem]{Proposition}
\newtheorem{proposition-definition}[theorem]
{Proposition-Definition}
\newtheorem{corollary}[theorem]{Corollary}
\newtheorem{lemma}[theorem]{Lemma}

\theoremstyle{definition}
\newtheorem{definition}[theorem]{Definition}

\newtheorem{example}[theorem]{Example}

\newtheorem{remark}[theorem]{Remark}

\theoremstyle{remark}

\newcommand\cA{\mathcal{A}} 
\newcommand\cC{\mathcal{C}} \newcommand\cD{\mathcal{D}}
 \newcommand\cF{\mathcal{F}}
 
\newcommand\cI{\mathcal{I}} 
\newcommand\cK{\mathcal{K}} \newcommand\cL{\mathcal{L}}
\newcommand\cM{\mathcal{M}} 
\newcommand\cO{\mathcal{O}} 
 
\newcommand\cS{\mathcal{S}} \newcommand\cT{\mathcal{T}}
\newcommand\cU{\mathcal{U}}

\newcommand\gr{\mathrm{gr}}

 \newcommand\DD{\mathbb{D}}

 \newcommand\PP{\mathbb{P}}

 \newcommand\ZZ{\mathbb{Z}}

\newcommand\OO{\mathcal{O}_C}

\newcommand\arr{\ifinner\to\else\longrightarrow\fi}

\newcommand\arrto{\ifinner\mapsto\else\longmapsto\fi}

\def\displaytimes_#1{\mathrel{\mathop{\times}\limits_{#1}}}

\def\displayotimes_#1{\mathrel{\mathop{\bigotimes}\limits_{#1}}}

\newcommand\pic{\operatorname{Pic}}

\newdir{ >}{{}*!/-5pt/@{>}}

\newcommand\doublelong[2]{\mathbin{\xymatrix{{}\ar@<3pt>[r]^{#1}
\ar@<-3pt>[r]_{#2}&}}}

\newlength{\ignora}

\newcommand{\lra}{\longrightarrow}

\newcommand{\ra}{\rightarrow}

\newcommand{\uu}{\mathcal{U}_C(r,0)} 

\newcommand{\su}{\mathcal{SU}_C(r)}

\newcommand{\urg}{\mathcal{U}_C(r,rg)}

\newfont{\sheaf}{eusm10 scaled\magstep1}

\begin{document}

\title[Coherent systems and subvarieties of $\mathcal{SU}_C(r)$]{Coherent systems and modular subvarieties of $\mathcal{SU}_C(r)$.}

\author[Bolognesi]{Michele Bolognesi}
\address{IRMAR\\263 Av. du Général Leclerc\\35042 Rennes\\France}
\email[Bolognesi]{michele.bolognesi@univ-rennes1.fr}

\author[Brivio]{Sonia Brivio}
\address{Dipartimento di Matematica "F.Casorati"\\Università di Pavia\\Via Ferrata 1\\
27100\\ Italy}
\email[Brivio]{sonia.brivio@unipv.it}


\begin{abstract}
Let $C$ be an algebraic smooth complex curve of genus $g>1$. The object of this paper is the study of the birational structure of  
certain moduli spaces of vector bundles and of coherent systems on $C$ and the comparison of different type of notions of stability arising in moduli theory. Notably we show that in certain cases these moduli spaces are birationally equivalent to fibrations over simple projective varieties, whose fibers are GIT quotients $(\PP^{r-1})^{rg}//PGL(r)$, where $r$ is the rank of the considered vector bundles. This allows us to compare different definitions of (semi-)stability (slope stability, $\alpha$-stability, GIT stability) for vector bundles, coherent systems and point sets, and derive relations between them. 
In certain cases of vector bundles of low rank when $C$ has small genus, our construction produces families of classical modular varieties contained in the Coble hypersurfaces.
\end{abstract}

\maketitle

\section{Introduction}

Let $C$ be a genus $g>1$ smooth complex algebraic curve, if $g\neq 2$ we will also assume that $C$ is non-hyperelliptic.

Let $\uu$ be the moduli space of rank $r$ semi-stable vector bundles on $C$ with degree zero determinant and let us denote as usual by $\su$ the moduli subspace given by vector bundles with trivial determinant. These moduli spaces appeared first in the second half of the last century thanks to the foundational works of 
Narashiman-Ramanan  \cite{ramanara} and Mumford-Newstead \cite{MumNews} and very often their study has gone along the study of the famous theta-map

\begin{eqnarray*}
\theta:\su & \dashrightarrow & |r\Theta|;\\
E & \mapsto & \Theta_E:=\{L\in \pic^{g-1}|h^0(C,E\otimes L)\neq 0\}.
\end{eqnarray*}

While we know quite a good deal about $\theta$ for low genera and ranks, as the rank or the genus grows our knowledge decreases dramatically, see Sect. 2 for a complete picture of known results.
The question of the rational type of these varieties is even more daunting. When rank and degree are coprime the situation is completely settled \cite{daking} but when the degree is zero (or degree and rank are not coprime) the open problems are still quite numerous. It is known that all the spaces $\su$ are unirational but the rationality is clear only for $r=2,\ g=2$, when in fact the moduli space is isomorphic to $\PP^3$, \cite{ramanara}. Some first good ideas about the birational structure for $g=2$ were developed in \cite{ang}. Then the $r=2$ case was analyzed in any genus by the first named author and A.Alzati in \cite{albol} with the help of polynomial maps classifying extensions in the spirit of \cite{ab:rk2}. In this paper we give a description for the higher rank cases via a new construction.

Our approach consists in studying the birational structure of $\uu$ and $\su$ via a study of similar moduli spaces of augmented vector bundles, notably the moduli space of coherent systems. By a coherent system on a curve we mean a vector bundle together with a linear subspace of given dimension of its space of global sections. Coherent systems come with a notion of stability that depends on a real parameter $\alpha$, that leads to a finite family of moduli spaces depending on the value of $\alpha$. Hence typically one will write $G_{\alpha}(r,d,k)$ for a moduli space of coherent systems, where $\alpha$ is the real parameter, $r$ is the rank of vector bundles, $d$ their degree and $k$ the prescribed dimension of the space of sections  (see Section 4 for details).
It turns out that for $\alpha >g(r-1)$ the moduli space $G_{\alpha}(r,rg,r)$ has a natural structure of a fibration and, moreover it is birational to $\uu$. The first main theorem of this paper is then the following. 

\begin{theorem}\label{global}
Let $C$ be a smooth complex curve of genus $g>1$, non-hyperelliptic if $g>2$, and let $\alpha >g(r-1)$. Then $G_{\alpha}(r,rg,r)$ is birational to a fibration over $C^{(rg)}$ whose fibers are GIT quotients $(\PP^{r-1})^{rg}//PGL(r)$.
\end{theorem}

Of course, since if $\alpha >g(r-1)$, $G_{\alpha}(r,rg,r)$ is birationally equivalent to $\uu$, a corresponding result holds also for $\uu$.
Notably, if we consider the moduli subspace $\su\subset\uu$ of vector bundles with fixed trivial determinant we get the following.

\begin{theorem}\label{globalsu}
The moduli space $\su$ is birational to a fibration over $\PP^{(r-1)g}$ whose fibers are GIT quotients $(\PP^{r-1})^{rg}//PGL(r)$.
\end{theorem}



Theorem  \ref{globalsu} allows us to give a more precise explicit description of the projective geometry of the fibration of $\su$ in the case $r=3,g=2$. In fact $\mathcal{SU}_C(3)$ is a double covering of $\PP^8$ branched along a hypersurface of degree six $\cC_6$ called the \textit{Coble-Dolgachev sextic} \cite{ivolocale}. Our result is the following.

\begin{theorem}
The Coble-Dolgachev sextic $\cC_6$ is birational to a fibration over $\PP^4$ whose fibers are Igusa quartics. More precisely, $\cC_6$ contains a 4-dimensional family of Igusa quartics parametrized by an open subset of $\PP^4$.
\end{theorem}

We recall that an Igusa quartic is a modular quartic hypersurface in $\PP^4$ that is related to some classical GIT quotients (see e.g. \cite{do:pstf}) and moduli spaces. Its dual variety is a cubic 3-fold called the \textit{Segre cubic}, that is isomorphic to the GIT quotient $(\PP^1)^6//PGL(2)$. 

\smallskip

If $r=2$ and $g=3$, then $\cS\cU_C(2)$ is embedded by $\theta$ in $\PP^7$ as a remarkable quartic hypersurface $\cC_4$ called the \textit{Coble quartic} \cite{ramanara}. Our methods also allow us to give a quick proof of the following fact.

\begin{proposition}
The Coble quartic $\cC_4$ is birational to a fibration over $\PP^3$ whose fibers are Segre cubics. There exists a 3-dimensional family of Segre cubics contained in $\cC_4$.
\end{proposition}

We underline that the cases of $\cC_4$ and $\cC_6$ are particularly interesting because one can interpret explicitly the beautiful projective geometry of the Igusa quartic and the Segre cubic in terms of vector bundles on $C$ (see Sect. 6). We hope that these results could help to shed some new light on the question of rationality of $\su$ and on the properties of the theta map.

\smallskip

On the other hand, a side result of Thm. \ref{global} and of its proof is that we get a bijection between the general vector bundle (or coherent system) and a set of $rg$ points in $\PP^{r-1}$. It then makes perfect sense to compare the GIT stability of a set of points in $(\PP^{r-1})^{rg}//PGL(r)$ with the slope stability of vector bundles and the $\alpha$-stability of coherent systems. This is discussed in Section 6.

\smallskip

A little caution: in some points of the paper it is important to distinguish a vector bundle $E$ from its S-equivalence class $[E]$. We have tried to keep the two distinguished notations when it is necessary, using just the vector bundle one when it has no importance, even if this sometimes may offend the good taste of the reader.

\smallskip

\textit{Acknowledgments:} The inspiration of this work came from a conversation that the first named author had with Norbert Hoffmann. This work also benefitted from remarks and discussions with Christian Pauly, Alessandro Verra and Angela Ortega. Arnaud Beauville brought to our attention some existing literature. Quang Minh and Cristian Anghel kindly let us have copies of their papers. Finally Duco Van Straten and Edoardo Sernesi gave us good advice about the deformation theory of singular projective hypersurfaces.

\medskip

\textbf{Description of contents.}

In section 2 we collect a few results on the theta map. In section 3 we outline the relation between theta maps and theta-linear systems by introducing the theta divisor of a vector bundle with integral slope. 
In section 4 we introduce generically generated coherent systems and their moduli spaces plus some properties of the evaluation and the determinant map for a coherent system of any rank. In Section 5 we introduce the \textit{fundamental divisor} of a coherent system. This definition allows us to define the \textit{fundamental map} in Section 6, the fibers of this map give us the fibration we look for. Then we prove Theorem \ref{global} and discuss briefly the relation between $\alpha-$stability of coherent systems, slope stability of vector bundles and GIT stability of point sets in the projective space. Finally, in Section 7 we restrict our analysis to moduli of vector bundles. Notably we apply our results to the cases $g=2,\ r=3$ and $g=3,\ r=2$ and construct explicit families of invariant hypersurfaces contained in the moduli spaces. 

\section{The Theta map.}\label{thetas}

Let $C$ be a smooth complex  algebraic curve of genus $g \geq 2$, we assume  that it is  non-hyperelliptic if $g>2$. 
Let $\pic^d(C)$ be the Picard variety  parametrizing  line bundles of  degree $d$  on $C$,  $\pic^0(C)$ will often be  denoted by $J(C)$. Let  $\Theta \subset \pic^{g-1}(C)$ be the canonical theta divisor   

$$\Theta:=\{L\in \pic^{g-1}(C)| \ h^0(C,L)\neq 0 \}.$$

For $r \geq 2$, let $\su$ denote the coarse moduli space of semi-stable vector bundles of rank $r$ and trivial determinant on $C$. It is a normal, projective variety  of dimension $(r^2-1)(g-1)$. 
It is well known that $\su$ is locally factorial and that $\pic (\su) = \ZZ$ \cite{dn:pfv}, generated by a line bundle
$\cL$ called the \textit{determinant bundle}. On the other hand, for $E\in \su$ we define

$$\Theta_E:=\{L\in \pic^{g-1}(C) | \ \ \ h^0(C,E\otimes L)\neq 0 \ \ \}.$$
\noindent This is either a divisor in the linear system $|r\Theta|$ or the whole $\pic^{g-1}(C)$. For $E$ a general  bundle  $\Theta_E$ is a divisor, {\it the theta divisor} of $E$.  This means that one can define the rational {\it theta map} of $\su$:

\begin{eqnarray}
\theta: \su  \dashrightarrow  |r\Theta|
\end{eqnarray}

\noindent sending $E $ to its theta divisor  $\Theta_E$.
The relation between the theta map and the determinant bundle  is given by the following fundamental result:

\begin{theorem}\cite{bnr}
There is a canonical isomorphism $|r\Theta| \stackrel{\sim}{\lra} |\cL|^*$ which  identifies $\theta$ with the rational map ${\varphi}_{\cL} \colon 
\su \dashrightarrow |\cL|^*$ associated to the determinant line bundle. 
\end{theorem}


The  cases when $\theta$ is a morphism or finite are of course very appealing. 
Notably,  $\theta$ is an embedding for $r=2$ \cite {ramanara}, \cite{BVtheta},\cite{vgiz} and it is a morphism when $r=3$ for $g \leq 3$ and  for a general curve of genus $g > 3$, \cite{bove23}, \cite{raysect}. Finally, $\theta$ is generically finite for $g=2$  \cite{bove23}, \cite{briverraBN} and we know its  degree  for $r \leq 4$,  \cite{ivolocale}, \cite{deg16}. 
There are also good descriptions of the image of $\theta$ for $r=2 \quad g=2,3$ \cite{rana:cra} \cite{paulydual},  $r=3,\ g=2$ \cite{orte:cob} \cite{Quang},  $r=2,\ g=4$ \cite{oxpau:heis}.  
Moreover, it has recently been shown in \cite{briverraPL} that if $C$ is general and $ g >> r$ then $\theta$ is generically injective.

\section{Vector bundles and theta linear systems}\label{thetasystems}

The notion of theta divisor can be extended to  vector bundles  with integral slope, in this paper we will consider 
 bundles with slope  $g = g(C)$.  
Let $\urg$ be the  moduli space of semi-stable  vector bundles  as above.
The tensor product defines a natural map:

\begin{eqnarray*}
t \colon  \su \times \pic^{g}(C) & \ra & \urg;\\
(E, \cO_C(D)) & \mapsto  &  E \otimes {\cO}_C(D), 
\end{eqnarray*}

\noindent which is étale, Galois, with Galois group $J(C)[r]$, the group of $r$-torsion points of the Jacobian of $C$. 

   
Moreover, if one restricts $t$ to $\su \times \cO_C(D)$ this yields an isomorphism $t_{D}  \colon \su \to  \mathcal{SU}_C(r,\cO_C(rD)),$ where the latter is the moduli space of rank $r$ semi-stable vector bundles with determinant  $\cO_C(rD)$.

\begin{definition}\label{supertheta}
Let $F \in \urg$, then we define the {\it theta divisor } of $F$ as follows:

\begin{equation*}
\Theta_F:=\{L \in \pic^1(C) | \ h^0(C,F \otimes L^{-1} )\neq 0\}.
\end{equation*} 

\end{definition}

Let $E \in \su$ and $\cO_C(D) \in \pic^g(C)$. If $F = E \otimes \cO_C(D)$, then we have that $\Theta_F = \cO_C(D) - {\Theta}_E$, thus $\Theta_F$ is a divisor if and only  ${\Theta}_E$ is a divisor.  We define 

\begin{equation} 
{\Theta}_D \colon = \{ L \in \pic^1(C) \vert \ \ h^0(\cO_C(D) \otimes L^{-1}) \geq 1 \ \}.
\end{equation}

Then, for any $r \geq 1$,  we have a natural isomorphism $\sigma_D:|r\Theta| \ra |r\Theta_D|$ given by the translation
$M  \mapsto \cO_C(D) - M$; moreover, if $\cO_C(rD_1) \simeq \cO_C(rD_2)$, then $|r\Theta_{D_1}| = |r\Theta_{D_2}|$.  
So we conclude that if $F  \in \urg$ admits a theta divisor, then ${\Theta}_F  \in \vert r {\Theta}_D \vert$, for any line bundle $\cO_C(D) \in  \pic^g(C)$ which is a $r^{\mathrm{th}}$-root of $\det F$. 
In this way we obtain a family of theta linear systems over the Picard variety $\pic^{rg}(C)$, as the following shows.




\begin{lemma} 
There exists a projective bundle ${\mathcal T}$  over $\pic^{rg}(C)$:
$$ p \colon {\mathcal T} \to \pic^{rg}(C), $$
whose fiber over $\cO_C(M)\in \pic^{rg}(C)$  is  the linear system $ \vert r {\Theta}_D \vert$, 
where $\cO_C(D)\in \pic^{g}(C) $ is any $r^{\mathrm{th}}$-root of $\cO_C(M)$.

\end{lemma}

\begin{proof}
Remark first that the linear system  $|r\Theta_D|$ is well-defined since it does not depend on which $r^{\mathrm{th}}$-root $\OO(D)$ of $\OO(M)$ we choose. Let us now consider the following  map:

\begin{eqnarray*}
\delta \colon \pic^g(C) \times \pic^1(C) & \to & \pic^{g-1}(C);\\
(\cO_C(D), L) & \mapsto & \cO_C(D) \otimes L^{-1}.
\end{eqnarray*}

For any $\cO_C(D) \in \pic^g(C)$,  we have: ${\delta}^* {r\Theta}_{|\cO_C(D) \times \pic^1(C)} \simeq \cO_{\pic^1(C)} ({r\Theta}_D).$ Let $p_1 \colon  \pic^g(C) \times \pic^1(C) \to \pic^g(C)$ be the projection onto the first factor. Consider the sheaf $ \cF:={p_1}_{*} \cO_{\pic^g(C) \times \pic^1(C)}( {\delta}^*(r\Theta)).$ It is locally free and its fiber at $\cO_C(D) \in \pic^g(C)$ is canonically identified with  the following vector space 
$$H^0( \pic^1(C), \cO_{\pic^1(C)} (r{\Theta}_D)).$$
Let $\widetilde{{\mathcal T}}$ be the projective bundle $\PP (\cF)$ on $\pic^g(C)$. 

Moreover the vector bundle $\cF$ is $J[r](C)$-equivariant, hence by easy descent theory (see \cite{stacchi} Thm. 4.46) it passes to the quotient by $J[r](C)$, i.e. the image of the cover $\rho  \colon \pic^g(C) \to \pic^{rg}(C)$ given by taking the $r^{th}$ power of each $L\in \pic^g(C)$. The projectivization of the obtained bundle is the projective bundle $\cT$ we are looking for.
We denote by  $p:\cT \ra \pic^{rg}(C)$  the natural projection on the base of the projective bundle.


\end{proof}

The previous arguments allow us to define the rational {\it theta map} of $\urg$.

\begin{eqnarray}\label{theta}
\theta_{rg} \colon \urg &  \dashrightarrow  &  {\mathcal T};\\
F & \mapsto & \Theta_F.
\end{eqnarray}

Finally, let us denote by $\theta_D$ the restriction of $\theta_{rg}$ to  $\mathcal{SU}_C(r,\cO_C(rD))$. Then we have the following commutative diagram:

$$\xymatrix{ \su \ar[r]^(.4){t_{D}} \ar@{-->}[d]_{\theta} & \mathcal{SU}_C(r,\cO_C(rD))  \ar@{-->}[d]^{{\theta}_D} \\
\vert r \Theta \vert \ar[r]^{{\sigma}_{D}} & \vert r {\Theta}_{D} \vert \\}$$

since $t_{D}$ and ${\sigma}_{D}$ are isomorphisms, we can identify the two theta maps. 
Finally remark that the composed  map $p \circ \theta_{rg}$ is precisely the natural map  which associates to each vector bundle $F$ its determinant line bundle $det (F)$.

\section{Generically generated coherent systems}
A pair $(F,V)$ is a { \it  coherent  system} of type $(r,d,k)$ on the curve $C$ if $F$ is a vector bundle of rank $r$ and degree $d$ on $C$ and  $V \subseteq  H^0(F)$ is a linear subspace of dimension $k$. A coherent system $(F,V)$ is {\it generically generated} if 
the evaluation map
$$ev_{F,V} \colon V \otimes \cO_C \to F, \quad (s,x) \to s(x)$$
has torsion cokernel. 
A   proper coherent subsystem of $(F,V)$ is a pair $(G,W)$  where $G  $ is a  non-zero sub-bundle  of $F$ and   $W \subseteq  V \cap H^0(F)$, with $(G,W) \not= (F,V)$.
\hfill\par
For any real number $\alpha$,  we define the $\alpha$-slope of a coherent system  $(F,V)$ of type $(r,d,k)$ as follows:
$$ \mu_{\alpha}(F,V) =  \frac{d}{r} + \alpha \frac{k}{r}.$$

\begin{definition}\label{alfa}
 A  coherent system   $(F,V)$ is $\alpha-$\textit{stable} (resp. $\alpha-$\textit{semi-stable}) if for any  proper coherent 
 subsystem $(G,W)$ of $(F,V)$ we have:
 $$ \mu_{\alpha}(G,W) < \mu_{\alpha}(F,V) \ \ ( \text{resp.} \leq ).$$
 \end{definition}

Every $\alpha-$semi-stable coherent system  $(F,V)$ admits a Jordan-H\"{o}lder fibration:

$$ 0= (F_0,V_0)  \subset (F_1, V_1) \subset .... \subset (F_n, V_n) = (F,V),$$

\noindent with  $  \mu_{\alpha}(F_j,V_j) = \mu_{\alpha}(F,V)$, $\forall j=1,..,n$ and s.t. each  coherent quotient system

 $$(G_j,W_j) = {{(F_j,V_j)} \over {(F_{j-1},V_{j-1})}}$$

\noindent is $\alpha-$stable. This defines the graded coherent system:

 $$ \gr(F,V) = {\bigoplus}_{j=1}^{n} (G_j,W_j).$$

Finally, we say that two $\alpha-$semi-stable coherent systems  are $S-equivalent$ if and only if their graded coherent systems are isomorphic.
\hfill\par
Let  $G_{\alpha}(r,d,k)$  be the moduli space  parametrizing  $\alpha-$stable coherent systems   of type $(r,d,k)$, its compactification  \ $\bar{G}_{\alpha}(r,d,k)$ is  a  projective scheme  parametrizing S-equivalence classes of $\alpha-$semi-stable coherent systems of type $(r,d,k)$,
see \cite{kingnew} for details. For $k\geq 1$, it follows easily from the definitions that, if $\widetilde{G}_{\alpha}(r,d,k)\neq \emptyset$, then $\alpha\geq 0$ and $d\geq 0$; if $G_{\alpha}(r,d,k)\neq \emptyset$, then $\alpha>0$.

\begin{definition}\label{critic}
A positive real number $\alpha$ is said to be a { \it virtual critical value} for coherent systems of type $(r,d,k)$ if it is numerically possible for 
a system $(F,V)$ to have a proper subsystem $(G,W)$ of type $(r',d',k')$ such that $\mu_{\alpha}(F,V) = \mu_{\alpha}(G,W)$ with $\frac{k}{r} \not= \frac{k'}{r'}$. If there is a coherent system $(F,V)$ and a
subsystem $(F',V')$ such that this actually holds, we say that $\alpha$ is an \textit{actual critical value}.
\end{definition}

It is well known (\cite{BGPMN}, Sect. 2.1 and 4) that, for coherent systems  of type $(r,d,k)$, the actual critical values form a finite set:

$$ 0 = {\alpha}_0  < {\alpha}_1 <... < {\alpha}_L,$$

\noindent and that within the interval $({\alpha}_i, {\alpha}_{i+1})$ the property of $\alpha-$stability of a pair is independent of $\alpha$. This means that $G_{\alpha}(r,d,k)$ is isomorphic to $G_{\alpha'}(r,d,k)$ whenever $\alpha$ and $\alpha'$ are contained in the same open interval $(\alpha_i,\alpha_{i+1})$. The same isomorphism holds for the respective compactifications. It is customary to call $G_L(r,d,k)$ the \textit{terminal} moduli space, i.e. the one that comes within the range $(\alpha_L, +\infty)$.

\smallskip

In this paper we will consider coherent systems of type $(r,rg,r)$. The following properties hold, thanks to \cite{BGPMN}, Thms. 4.4, 4.6 and 5.6.

\begin{proposition}\label{terminal}
 Let $r \geq 2$. For   $\alpha  > g(r-1)$  the moduli spaces $G_{\alpha}(r,rg,r)$ stabilize, i.e. we have: 
$$G_{\alpha}(r,rg,r)= G_{L}(r,rg,r), \quad   {\text if} \ \ \alpha  > g(r-1).$$ 
The moduli space $G_L(r,rg,r)$  is a smooth quasi-projective variety of dimension $r^2(g-1) + 1$ and its compactification $\bar{G}_{L}(r,rg,r)$ is irreducible. Moreover, after $g(r-1)$, each  $\alpha-$semi-stable $(F,V)  $ is  generically generated.
\end{proposition} 
 
One relation between the stability of a coherent system and that of the underlying vector bundle is given by the following:

\begin{lemma}\label{props}
Let $(F,V)$ be a coherent system of type $(r,rg,r)$,  which is generically generated. Then we have the following properties:
\begin{enumerate}
\item{} If $F$ is semi-stable, then $(F,V)$ is $\alpha-$semi-stable for any $\alpha\geq 0$;
\item{} if $F$ is stable,  then $(F,V)$ is $\alpha-$stable for any $\alpha > 0$;
\item{} if $\alpha >  g(r-1)$, then either $(F,V)$ is $\alpha-$semi-stable or  there exists a subsystem $(G,W)$ of type $(s,d,k)$ with $s=k$ and  
$\frac{d}{s} > g$.   
\end{enumerate}
\end{lemma}

\begin{proof}
Let $(G,W)$ be a proper coherent subsystem of $(F,V)$ of type $(s,d,k)$:
$$1 \leq s \leq r,  \quad  0 \leq k \leq r, \quad  (G,W) \not= (F,V).$$
We show that   $(F,V)$  generically generated implies   $\frac{k}{s}  \leq  1$. 
We can consider  the map ${ev_{F,W}} \colon W \otimes \cO_C \to  F$,  its image is a subsheaf $Im (ev_{F,W})$ of $F$. The inclusion $W \subset V$ implies  $Im(ev_{F,W})$ has  generically rank $k$. 
Finally,  $W \subset H^0(G)$ so  $Im (ev_{F,W})$ is a subsheaf of $G$,   which implies $k \leq s$. 
\hfill\par
(1)-(2) Suppose that $F$ is  semi-stable. Since for any proper coherent subystem $(G,W)$ of type $(s,d,k)$ we have  $\frac{k}{s}  \leq 1$  then   for any $\alpha$ we have:    $\mu_{\alpha}(G,W)  \leq  \mu_{\alpha}(F,V)$. 
In particular, if $F$ is stable, then it is easy to see that $(F,V)$ is  $\alpha$-stable if $\alpha > 0$.
\hfill\par
(3) Let  $(G,W)$ be a coherent subsystem  of type $(s,d,k)$ which contradicts the $\alpha-$semi-stability of 
$(F,V)$. As we have seen, $\frac{k}{s}  \leq  1$. If $\frac{k}{s} = 1$, then $\frac{d}{s} > g$. If $F$ is not semi-stable and 
$(F,V)$ is not $\alpha-$semi-stable. If $\frac{k}{s} < 1$, then $\frac{F}{G}$ is generically generated too, hence $deg(\frac{F}{G}) \geq 0$.
This implies $d  \leq  rg$. So we have:

$$ g + \alpha  <  \frac{d}{s} + \alpha \frac{k}{s} \leq \frac{rg}{s} + \alpha \frac{k}{s},$$

\noindent which implies:

$$\alpha  <   \frac{g(r-s)}{s-k} \leq g(r-1).$$

\end{proof}

\begin{lemma} There is a natural birational   map $b \colon \urg \dashrightarrow {\bar G}_L(r,rg,r).$ 
\end{lemma}

\begin{proof}
Let $U_{r,rg} \subset \urg$ be the subset of stable points $[F]$  satisfying 
 the following properties:

\begin{enumerate}
\item{}  $ h^0( F) = r$;
\item{} the determinant map $d_{F} \colon {\wedge}^r H^0(F) \to H^0(det F)$ is not the zero map.
\end{enumerate}

The conditions that define the set $U_{r,rg}\subset \urg$ are clearly open; we prove that $U_{r,rg}$ is not empty.
  Let $F_0 = L_1 \oplus L_2 \oplus .. \oplus L_r$, with $L_i$ a line bundle of degree $g$ and $h^0(L_i) = 1$ for any $i =1,.,r$.  Then $F_0$ satisfies  (1) and (2). Let $F_t$ be a stable deformation of $F_0$ along a one parameter family $T$. For a generic $\bar{t}\in T$, $F_{\bar{t}}$ satisfies (2). Moreover, by semicontinuity $h^0(F_{\bar{t}}) \leq r$ and
  by Riemann-Roch $h^0(F_{\bar{t}}) \geq r$, so $F_{\bar{t}}$ satisfies (1) too,  hence $[F_{\bar{t}}] \in U_{r,rg}$. 
\hfill\par
Let $[F] \in U_{r,rg}$.   
The pair $(F, H^0(F))$ is a coherent system of type $(r,rg,r)$, by Property (2) it is generically generated and  by Lemma \ref{props} (2)  it is $\alpha-$stable. So we have a morphism 
\begin{eqnarray}\label{bi}
b \colon  & U_{r,rg} \to  & {\bar G}_L(r,rg,r);\\
{[F]} & \to & {[(F,H^0(F))]}, 
\end{eqnarray}

it is easy to see that it  is an isomorphism onto the image.  Since   $\dim \urg = \dim {\bar G}_L(r,rg,r)$ and ${\bar G}_L(r,rg,r)$ is irreducible, we conclude that $b$ induces a birational map between the moduli spaces  $\urg$ and ${\bar G}_L(r,rg,r)$.
\end{proof}

\section{The fundamental divisor of a coherent system} 

\rm
Let $(F,V)$ be an $\alpha-$stable coherent system  of type $(r,rg,r)$ on the curve $C$. Assume that  it is generically generated, 
then the  map $ev_{F,V} \colon V \otimes \cO_C \to F$ is a generically surjective map between two vector bundles of the same rank, so its degeneracy locus is an effective divisor on the curve $C$:

$${\mathbb D}_{(F,V)} \in \vert \det F \vert .$$

Moreover, let us consider the restriction to $\wedge^rV$ of the determinant map of $F$  

\begin{eqnarray*}
d_{F,V} \colon {\wedge}^rV & \lra & H^0(C,\det F);\\
s_1 \wedge s_2 ... \wedge s_r & \mapsto & ( x \mapsto s_1(x)  \wedge s_2(x)  \wedge... \wedge s_r(x) ).
\end{eqnarray*}
\noindent
This is not zero so its image is a line generated by a non-zero global section $\sigma$ of $H^0(\det F)$:

$$d_{F,V}({\wedge}^r V)= \text{Span} (\sigma),  \quad \sigma \in H^0(C,det F).$$
\noindent
It is easy to see that ${\mathbb D}_{(F,V)} = \text{Zeros} (\sigma).$

\begin{definition}
We call ${\mathbb D}_{(F,V)}$ the {\it fundamental divisor} of   $(F,V)$. 
\end{definition}

\hfill\par
Let  $(F,V) = (F_1,V_1) \oplus (F_2,V_2)$, assume that for $i=1,2$, the pair $(F_i,V_i)$ is an $\alpha-$stable coherent system, which is generically generated too.  Then we have $ {\mathbb D}_{(F,V)}= {\mathbb D}_{(F_1,V_1)} + {\mathbb D}_{(F_2,V_2)}.$

\begin{definition}
Let $[F,V]$ be the S-equivalence class of an $\alpha-$semi-stable coherent system $(F,V)$. We define the \textit{fundamental divisor of} $[F,V]$ as

$${\mathbb D}_{[F,V]}:= {\mathbb D}_{\gr(F,V)}.$$
\end{definition}

Note that if $(F,V)$ is $\alpha-$stable, then $\gr(F,V) = (F,V)$, so that $ {\mathbb D}_{[F,V]}= {\mathbb D}_{(F,V)}$.
\newcommand{\xrg}{C^{(rg)}}
\noindent
Every coherent system $(F,L)\in {\bar G}_L(r,rg,r)$ is generically generated (see Prop. \ref{terminal}), hence we can define the following map.

\begin{eqnarray}
\Phi \colon {\bar G}_L(r,rg,r) & \ra & C^{(rg)}\\ 
{[F,V]} & \mapsto & \mathbb{D}_{\gr(F,V)}.
\end{eqnarray}

\begin{definition}
We call $ \Phi$ the {\it fundamental map} of generically generated coherent systems of type $(r,rg,r)$.
\end{definition}

\section{The  fundamental map and its fibers}\label{secdet}

The aim of this section is the description of the fibers of $\Phi$. We start by showing some basic properties of the map $\Phi$ itself.

\begin{theorem}\label{domino}
For any $r \geq 2$,  $\Phi \colon {\bar G}_L(r,rg,r) \to  C^{(rg)}$ is a surjective morphism. 
\end{theorem}

\begin{proof}
 Let $({\mathcal F}, {\mathcal V})$ be a flat family of $\alpha$-semi-stable generically generated coherent systems of type $(r,rg,r)$ over a scheme $S$.
Then ${\mathcal F}$ is a rank $r$ vector bundle on $C \times S$, ${\mathcal F}_{\vert C \times s} = F_s$ is a 
 vector bundle   of rank $r$ and degree $rg$ on $C$. Let $p_1$ and $p_2$ be the natural projections of $C \times S$ onto its factors,   ${\mathcal V} \subset (p_2)_*{\mathcal F}$ is a vector bundle 
 of rank $r$ on $S$,  with  fiber  $V_s$ at the point $s$.    Finally, $(F_s, V_s)$ is a generically generated coherent system of type $(r,rg,r)$ on the curve $C$, which is $\alpha-$semi-stable, for any $s \in S$. Let $\mu \colon S \to {\bar G}_L(r,rg,r)$ be the  morphism defined by the family,  sending  $s \to (F_s, V_s)$. 
 \hfill\par
 Let $E_{v} \colon (p_1)^* {\mathcal V} \to {\mathcal F}$ be the natural evaluation map, note that it is a  map between two vector bundles of the same rank, moreover   for any $s \in S$ we have:  ${E_{v}}_{\vert C \times s}= ev_{F_s,V_s}$. This implies that  the degeneracy locus of $E_v$ is a relative divisor ${\mathcal D}$ on $C$ over $S$ of relative degree $rg$. For any $s \in S$ we have ${\mathcal D}_{\vert C \times s} = {\mathbb D}_{F_s,V_s}$. 
This induces a morphism ${\Phi}_S \colon S \to C^{(rg)}$, sending $s \to {\mathbb D}_{F_s,V_s}$, such that     ${\Phi}_S = {\Phi} \circ \mu$.
This proves that $\Phi$ is a morphism.
\hfill\par
Let us come to surjectivity. Let  $G$ be a point set in $C^{(rg)}$,  note that $G$ can be written as the sum of $r$ effective divisors of degree $g$:

$$G= G_1 + G_2 +... + G_r.$$ 

For any $i =1,..,r$, let ${\sigma}_i \in   H^0(\cO_C(G_i))$ be a non-zero global section  such that $G_i = \text{Zeros}({\sigma}_i)$ and let 
$V_i = \text{span} ({\sigma}_i ) \subset H^0(\cO_C(G_i))$. 
Now let us define the following pair:

 $$F \colon = \bigoplus_{i=1}^r \cO_C(G_i) \quad V \colon = \bigoplus_{i=1}^r V_i.$$
\noindent
Then $(F,V)$ is a generically generated coherent system of type $(r,rg,r)$ on $C$ and $F$ is semi-stable. By Lemma \ref{props} (1), $(F,V)$ is $\alpha-$semi-stable and $\mathbb D_{F,V} = G$.
This implies that $\Phi([F,V]) = G$ and proves the surjectivity.
\end{proof}

Let $({\PP}^{r-1})^{rg} // PGL(r)$ denote  the GIT quotient of $({\PP}^{r-1})^{rg}$ with respect to the diagonal action of $PGL(r)$. We recall that the notion of GIT stability of a point set  $ v \in ({\PP}^{r-1})^{rg}$ has a nice geometric formulation in terms of the dimension of the linear span in $\PP^{r-1}$ of subsets of $v$ ( see \cite{do:pstf} Thm. 1 p. 23).

\begin{proposition}
The point set $v= (v_1,\dots,v_{rg})  \in ({\PP}^{r-1})^{rg}$ is GIT semi-stable\newline ( resp. stable ) if and only if for any subset $\{v_1,\dots,v_k\}$ of $v$ we have 

$$ \dim(  Span(v_1,\dots,v_k )) \geq \frac{k}{g}  \ \ ( \text{resp.} > ).$$
\end{proposition}
 
The first important feature of the map $\Phi$ is the following.

\begin{theorem}\label{genfib}
The general fiber of $\Phi$ is isomorphic to $({\PP}^{r-1})^{rg} // PGL(r)$.
\end{theorem}

\begin{proof}
Let  $B \in C^{(rg)}$, we assume  that $B$ is not contained in the big diagonal $\Delta$, that is 

$$B = \sum_{i=1}^{rg}{x_i}, \quad x_i \not= x_j,  \quad  \forall  i \not= j, \quad x_i\in C, \forall i.$$
\noindent
Then the fiber of $\Phi$ at $B$ is the following:

$${\Phi}^{-1}(B) =  \{ [(F,V)] \in    {\bar G}_{L}(r,rg,r) \  \vert \  {\mathbb D}_{\gr(F,V)} = B \ \}.$$ 
\noindent
Let  $[F,V] \in {\Phi}^{-1}(B)$, from now on we will write for simplicity $\gr(F,V) := (F_g,  V_g)$. Since $B$ is the degeneracy locus of the evaluation map $ev_{ F_g, V_g}$, by dualizing we find the following exact sequence:

\begin{equation}\label{sequa2}
0 \to F_g^* \to V_g^* \otimes \cO_C \to \cO_B \to 0.
\end{equation}
\noindent
Hence, up to the choice of a basis of $V_g$, $F_g^*$ is the kernel of a surjective morphism $v_{F_g,V_g} \in \mathrm{Hom}(V_g^* \otimes \cO_C,\cO_B):$

$$ v_{F_g,V_g}= (v_1, \dots , v_{rg}), \quad v_i \not= 0, \quad v_i \in  \mathrm{Hom}(V_g^* \otimes \cO_C,\cO_{x_i}) \simeq   V_g.$$ 
\noindent
This means that $(F_g,V_g)$ defines a point (that we will still denote by $v_{F_g,V_g}$) in the  product space  $({\PP (V_g)})^{rg}\cong({\PP}^{r-1})^{rg}$, and $PGL(r)$ acts diagonally on this space via the choice of the basis of $V_g$.

\medskip

{\it Claim 1: $v_{F_g,V_g} \in ({\PP}^{r-1})^{rg}$ is GIT semi-stable with respect to the diagonal action of $PGL(r)$. If $(F_g,V_g)$ is $\alpha-$stable then $v_{F_g,V_g}$ is GIT stable}

\smallskip

\begin{proof}
Let  $w=\{v_1,\dots,v_d\}$ be a subset of $v_{F_g,V_g}$, set $W \colon =Span(v_1,\dots,v_d)$.  Let $x_i$ be the point of $B$ that corresponds to $v_i$, for 
$i =1,..,d$.   Then $W \subset V_g$,  so we get a commutative diagram as follows:

$$\xymatrix{ 0  \ar[r] &  F_g^* \ar[d]   \ar[r]  & V_g^* \otimes \cO_C   \ar[d]    \ar[r]^{v_{F_g,V_g}} & \cO_B  \ar[d] \ar[r] & 0  \\
0  \ar[r]  &  G^*  \ar[r]  &   W^* \otimes \cO_C   \ar[r]^{w}  &   \oplus_{i=1}^d\cO_{x_i}  \ar[r] & 0 \\ }$$ 

\noindent for some vector bundle $G^*$ with  $rk(G^*)=dim (W) = s$. Note that the pair $(G,W)$ is a coherent system  of type $(s,d,s)$, which is a proper  subsystem of  $(F_g,V_g)$. Since this is $\alpha-$semi-stable, we have:

$$ \mu_{\alpha}(G,W) = \frac{d}{s} + \alpha \leq  \mu_{\alpha}(F_g,V_g) = g + \alpha,$$

\noindent which implies 
$s \geq \frac{d}{g}$, which gives  the GIT semi-stability of $v_{F_g,V_g}$. 
The stable case is described in the same way but with strict inequalities.
\end{proof}

Let $V$ be a vector space of dimension $r$ and $\PP(V)$ the associated projective space. By mimicking sequence (\ref{sequa2}), we can construct a flat family of coherent systems of type $(r,rg,r)$ on $C$ over $(\PP(V))^{rg}$.
Let  $v=(v_1,\dots,v_{rg}) \in \PP(V)^{rg}$,   $v$ defines  a  surjective morphism of sheaves $V^* \otimes \cO_C \ra \cO_B$ as follows:  it is the zero map out of the support of $B$ and it  is obtained by taking one lift of $v_i$ to $V^*$ and applying it on the fiber of $V ^*\otimes \cO_C$ over $x_i \in B$. The morphism  depends on the choice of the lift but the kernel of the sequence

\begin{equation}\label{univ}
0\lra \ker(v) \lra V^*\otimes \cO_C \stackrel{v}{\lra} \cO_B \to 0,
\end{equation}

is well defined over $\PP(V)^{rg}$. This implies that $F_v:=\ker(v)^*$, for  $v \in {\PP(V)}^{rg}$, is  a family of   rank $r$ vector bundles on $C$ with determinant $\cO_C(B)$. Note that $F_v$ is generically generated by a linear subspace of $H^0(C,F_v)$ of dimension $r$:  we will denote it by $V_v$. The pair $(F_v, V_v)$ is a coherent system of type $(r,rg,r)$, generically generated  and
$\mathbb D_{F_v,V_v} = B$.  Moreover the family $(F_v, V_v)$  is invariant under the diagonal action of $PGL(r)$ on ${\PP(V)}^{rg}$

\medskip

{\it Claim 2: let  $v \in \PP(V)^{rg}$ be  GIT semi-stable (resp. stable), then  the pair $(F_v,V_v)$ is $\alpha-$semi-stable (resp. stable) for $\alpha  >   g(r-1)$, hence $[F_v,V_v] \in {\bar G}_L(r,rg,r)$}.

\smallskip

\begin{proof} Let $\alpha > g(r-1)$. Since $(F_v,V_v)$ is generically generated, we can apply  Lemma \ref{props} (3): either $(F_v,V_v)$  is $\alpha-$semi-stable, or
there exists a proper  subsystem $(G,W)$  of type $(s,d,k)$ with $s = k$ and $\frac{d}{s} > g$.  
Then $(G,W)$ is generically generated too, and we have a commutative diagram:

$$\xymatrix{   0  \ar[r] & W \otimes O_C  \ar[d]   \ar[r]  & G \ar[d]    \ar[r]^{{v_{G,W}}} & \oplus_{i=1}^d\cO_{x_i}   \ar[d] \ar[r] & 0  \\
 0  \ar[r] & V \otimes \cO_C   \ar[r] & F_v  \ar[r]^{{v = v_{F_v,V}}}   &   \c O_B  \ar[r] & 0 \\
 }$$ 

Let $v_{G,W}=(v_1,.,v_d)$ and  $Span (v_1,..,v_d ) \subset W \subset V$  with 
$\dim W = s$. Since $v \in \PP(V)^{rg}$ is GIT semi-stable then we have that $\dim W  \geq \frac{d}{g}$, 
which  contradicts the hypothesis. This proves that $(F_v,V_v)$ is $\alpha-$semi-stable.\\ 
Suppose now that we have a stable point set $v$, and that there exists a proper coherent subsystem $(G,W)\subset (F_v,V_v)$ of type $(s,d,k)$ s.t. $\mu_{\alpha}(G,W)=\mu_{\alpha}(F_v,V_v)$. Since $\alpha$ is not a critical value, $\frac{k}{s}=1$, and thus $s=\frac{d}{g}$ (see Def. \ref{critic}). But, keeping the same notation of the first part of the proof, if $v$ is stable, then $dim(W)=s>\frac{d}{g}$. Hence $(F_v,V_v)$ is $\alpha$-stable.
\end{proof}

\smallskip

If $\alpha>g(r-1)$, Claim 2 and the fact that the family of coherent systems over $\PP(V)^{rg}$ is $PGL(r)$-invariant allow us to define a morphism 

\begin{eqnarray}\nonumber
\PP(V)^{rg}//PGL(r) & \rightarrow &  {\Phi}^{-1}(B).\\ \nonumber
\end{eqnarray}

Indeed, the family over $\PP(V)^{rg}$ induces a morphism from the semi-stable locus to ${\bar G}_L(r,rg,r)$. This morphism is $PGL(r)$-invariant hence it factors through $\PP(V)^{rg}//PGL(r)$. Furthermore, by construction its inverse is the map sending each system  $[F,V] \mapsto [v_{F_g,V_g}]$ (see Claim 1). Hence we conclude that the general fiber of $\Phi$ is isomorphic to $\PP(V)^{rg}//PGL(r)$. 
\end{proof}

\medskip
Let us denote $\alpha_1>0$ the smaller positive critical value for coherent systems of a given type $(n,d,k)$. The stability condition induced by values of $\alpha$ s.t. $0 < \alpha < \alpha_1$ is called $0^+$-stability (see \cite{hidalgos} page 4-5). It is well known that if a coherent system $(F,V)$ is $0^+$-stable then the underlying vector bundle $F$ is slope-semistable.
On the other hand, if the underlying vector bundle $F$ is stable, then $(F,V)$ is $0^+$-stable. In the following example,  we suppose $r=2$ and $g(C)>3$.  We  produce  a coherent system on $C$ of type $(2,2g,2)$ which comes from a GIT stable point and is $\alpha$-semi-stable for any $\alpha \geq 1$, but the underlying vector bundle is unstable, i.e. the coherent system is not $0^+$-semi-stable.

\smallskip

\begin{example}\label{destable} 
Let $C$ be a smooth  curve of genus $g >3$.  Let $L_1$ be a special line bundle  of degree $g-1$ with   $h^0(C,L_1)=2$. Let  $ \{ t_1, t_2 \}$ be  a basis of $H^0(C,L_1)$. We can assume that the zeros of $t_1$ are all simple and distinct. By definition, given a scalar $\lambda$, one may find at most a finite set (possibly empty) of points $x \in C$ such that $t_1(x)=\lambda t_2(x)$. Hence we are allowed to choose a set of $g+1$ distinct points $z_1,\dots, z_{g+1}$ such that:  $t_1(z_i) \neq  0$, $t_1(z_i)= \lambda_i t_2(z_i)$, with $\lambda_i \neq \lambda_j$ for $1 \leq i \neq j \leq g+1$ and $\lambda_i \neq 0$. 
Let $M:= z_1 + \dots + z_{g+1}$,  $L_2:=\cO_C(M)$ and let $s$ be a non-zero global section of  $H^0(C,L_2)$ such that $M = Zeros(s)$.
 Now let us consider the coherent system $(F,V)$ of type $(2,2g,2)$ defined as follows:

$$  F \colon = L_1 \oplus L_2 \quad V:= \langle t_1, t_2 + s \rangle.$$
\noindent
Note that the evaluation map

$$ ev_{F,V} \colon V \otimes \cO_C \longrightarrow F,$$
\noindent is generically surjective, degenerating on the divisor $B:=Zeros(t_1) + Zeros(s)$. 
\hfill\par\noindent
We show that  $(F,V) \in {\bar G}_{L}(r,rg,r)$, hence $(F,V) \in {\Phi}^{-1}(B)$, but the underlying vector bundle  is not semi-stable:  in fact $L_2 \subset F$ is a destabilizing sub-bundle.
Notably, we prove that $(F,V)$ is $\alpha$-semi-stable for any $\alpha \geq 1$. Let  $(G,W)$ be a coherent subsystem of 
type $(s,d,k)$. It is easy to see that  $\mu_{\alpha}(G,W) >  \mu_{\alpha}(F,V)$ implies that
$s =1$, $d = g+1$,   $k= 0$ and $\alpha < 1$. Note that $\alpha = 1$ is a critical value for coherent systems of type $(2,2g,2)$. 
\hfill\par
Finally,  let   $v = (v_1,v_2,\dots,v_{2g}) \in \PP(V)^{2g}$ be the point set defined by the pair $(F,V)$, as in the proof of Thm. \ref{genfib}. By the assumptions we made  choosing  $M$ and the sections $t_i$ and $s$, the point set $v$ has the following shape:
 for $i = 1,\dots,g-1$,  \  $x_i \in Zeros(t_1)$, so $[ v_i] = [1:0]$;   for $i = g,\dots,2g$, \  $x_i \in   Zeros(s)$, so  we have:
 $[v_i] = [\lambda_i : 1]$, with $\lambda_i \neq 0$.
Such a point set $v \in \PP(V)^{2g}$ is clearly stable.

\end{example}

\medskip

As the reader may expect, given the birationality result between the moduli space $\urg$ and $G_L(r,rg,r)$, the fundamental map gives information also on the geometry of $\urg$ and of the theta map. 

\begin{theorem}\label{thetabirat} 
Let $b$ be the birational map defined in (\ref{bi}):

$$ b  \colon \urg \dashrightarrow \bar{G}_L(r,rg,r), \quad [F] \to [ 
(F,H^0(F))].$$ 
\noindent
Let $b_{D}$ be its restriction to $\mathcal{SU}_C(r,\cO_C(rD))$ and 
let $\theta_{rg}$ be the theta map of $\urg$ defined in (\ref{theta}): 
and ${\theta}_D$ its restriction to $\mathcal{SU}_C(r,\cO_C(rD))$, 
then we have  the following  commutative diagrams of rational maps

$$\xymatrix{ \urg \ar@{-->}[r]^{b} \ar@{-->}[d]_{{\theta}_{rg}} & \bar{G}_L(r,rg,r)  \ar[d]^{\Phi} \\
\mathcal T \ar@{-->}[r]^{a^*} & C^{(rg)} \\}    \quad 
\xymatrix{  \mathcal{SU}_C(r,\cO_C(rD)) \ar@{-->}[r]^{b_D} \ar@{-->}[d]_{{\theta}_{D}} & b_D(\mathcal{SU}_C(r,\cO_C(rD)))  \ar[d]^{\Phi} \\
\vert r \Theta_D \vert \ar@{-->}[r]^{a^*} &  \vert \cO_C(rD) \vert  \\}$$

\noindent where $a^* $ is the pull-back of theta divisors via the  Abel map $a \colon C \to \pic^{(1)}(C)$. Notably, the restricted map
 
$${\pi}_e = {a^*}_{ | \vert r \Theta_D \vert} \colon  \vert r \Theta_D \vert  \to   \vert\cO_C(rD) \vert $$

\noindent is a linear projection with center $L_e := \{ M \in  \vert r \Theta_D \vert \  \colon \  a(C) \subset M \ \}$. 
\end{theorem} 
  
 \begin{proof}
Let $[F] \in \urg$. Then $[F]$ is contained in the regular locus of $\theta_{rg}$ if and only if $[F]$ admits a theta divisor

 $$ \Theta_{F} = \{ L \in \pic^{(1)}(C) \vert \ h^0(C,F \otimes L^{-1}) \geq 1 \}.$$

Furthermore, the divisor $\Theta_F$ lies in the regular locus of $a^*$ if it does not contain the image of $C$ via $a$.
Remark in fact that in this case, the pull back of $\Theta_F$ via $a^*$ is an effective divisor of degree $rg$. We have:

$$a^*(\Theta_F) = \{ x \in C \vert \ h^0(C,F \otimes \cO_C(-x)) \geq 1 \}.$$

This implies that $h^0(F) = r$ and the evaluation map $ev_{F,H^0(F)} \colon  H^0(F) \otimes \cO_C \to F$ is generically surjective, hence $b([F])$ is defined. 
In order to show that the two diagrams commute it is enough to remark that if $x\in \mathbb D_{(F,H^0(F))}$ then there exist at least one non-zero section in $h^0(C,F \otimes \cO_C(-x))$, hence $\mathbb D_{(F,H^0(F))}$ and $a^*(\Theta_F)$ coincide.
\end{proof}
 
As a corollary of the results of this section we have the following:

\begin{theorem} For any  $r \geq 2$ and $g \geq 2$, 
the moduli space $\urg$  is birational to a fibration over $C^{(rg)}$  whose fibers are $(\PP^{r-1})^{rg}//PGL(r)$ and
the moduli space $\mathcal{SU}_C(r,\cO_C(rD))$ is birational to a fibration over ${\mathbb P}^{(r-1)g}= \vert \cO_C(rD) \vert $  whose fibers are $(\PP^{r-1})^{rg}//PGL(r)$.
\end{theorem}

\section{ Application to the Coble hypersurfaces}

As Example \ref{destable} showed, in the case of vector bundles we cannot expect any isomorphism betweem the fibers of $\Phi$ and GIT quotients. Nevertheless, for low rank and genus, the geometry of moduli spaces of vector bundles is simple enough that, by taking the closure of the fibers, we still find, at least as projective varieties, families of GIT quotients contained in moduli spaces of vector bundles. This section is devoted to the construction and study of such families of projective varieties.

\medskip

When the genus of $C$ is 2 or 3 and the rank is small enough, the moduli spaces $\su$ have very nice explicit descriptions related to certain hypersurfaces, called the {\it Coble hypersurfaces}. The main theorems of this section show how these hypersurfaces contain large families of projective classical modular varieties related to invariant theory, namely the {\it Segre cubic} and the {\it Igusa quartic}. 

\smallskip

We know a purely vector bundle-theoretic construction (i.e. without the use of the theta map) of such families, but it is rather involved and, in our opinion, less instructive. For this reason and for the beauty of the objects of study, we have preferred to give a construction that relies on the theta map and the projective geometry of the Coble hypersurfaces.




\subsection{The Coble Sextic}\label{sextic}

In this subsection we assume that  $C$ is a curve of genus $2$ and we consider the moduli space $\cS\cU_C(3)$  of semi-stable vector bundles on $C$ with  rank $3$ and trivial determinant. The theta map 

$$\theta:\cS\cU_C(3) \ra |3\Theta| \simeq \PP^8$$ 

\noindent is a finite  morphism   of  degree $2$ and the branch locus is a sextic hypersurface $\cC_6$, called the Coble-Dolgachev sextic \cite{ivolocale}. The Jacobian variety $J(C)$ is embedded in $|3\Theta|^*$ as a degree 18 surface, and there exists a unique cubic hypersurface $\cC_3\subset |3\Theta|^*$ whose singular locus coincides with $J(C)$: {\it the Coble cubic}, \cite{coble}. 
It was conjectured by Dolgachev, and subsequently proved in \cite{orte:cob} and independently in \cite{Quang}, that $\cC_6$ is the dual variety of $\cC_3$. 
\hfill\par
Let us consider the moduli space  $\cS\cU(3,3K_C)$ and  let
 
$$\Phi_{K_C}: \cS\cU(3,3K_C) \dashrightarrow  |3K_C|, $$ 

\noindent be the rational map sending  $[F]$ to ${\mathbb D}_{\gr(F), H^0(\gr(F))}$. It is 
 the composition  of $b_{K_c}$ with the fundamental map $\Phi$.  
By Thm. \ref{thetabirat},  we have the following result:



\begin{proposition}\label{pitre}
 The map $\Phi_{K_C}$ is the composition of the theta map 
$$\theta_{K_C}: \cS\cU(3,3K_C)\ra |3{\Theta}_{K_C}|$$  with  the linear   projection ${\pi}_e \colon \vert 3 \Theta_{K_C} \vert \dashrightarrow \vert 3K_C \vert$, whose   center $L_e \simeq {\mathbb P}^3$,  is the linear subsystem of theta divisors corresponding to decomposable bundles of type $E \oplus K_C$, with $E\in \cS\cU_C(2,2K_C)$. 
\end{proposition}

\begin{proof}
 Since $dim \ |3K_C| \ =4$,  the center of the projection ${\pi}_e$  is a 3 dimensional linear subspace $L_e$, which is the image by 
 the theta map of the indeterminacy locus of $\Phi_{K_C}$. 
Let us  consider the subset  ${\mathcal I} \subset \cS\cU(3,3K_C)$ of  semi-stable bundles  $[F] = [E \oplus K_C]$, $[E ] \in \cS\cU_C(2,2K_C)$.
Note that ${\mathcal I} \simeq  \cS\cU_C(2,2K_C) \simeq {\mathbb P}^3$ \cite{rana:cra}. Moreover, since $h^0(\gr (E \oplus K_C) \geq  4$, then  ${\mathcal I}$  is contained in the indeterminacy locus of ${\Phi}_{K_C}$. 
Finally, let  $h$ be the hyperelliptic involution  on $C$, it defines the involution 

\begin{eqnarray*}
\sigma: \cS\cU(3,3\cK_C) & \ra &  \cS\cU(3,3\cK_C);\\
F & \mapsto & h^*F^*;
\end{eqnarray*}

which is associated  to the 2:1 covering given by $\theta_{K_c}$. Then ${\mathcal I}$ is contained in the fixed locus of $\sigma$,  \cite{Quang} (Sect. 3 and 4).  The image via $\theta_{K_C}$ of this locus is  $L_e \simeq {\PP^3}$ and  it is actually the center of the projection.
\end{proof}
\hfill\par
By Thm. \ref{thetabirat}, a general fiber of $\Phi_{K_C}$ is birational to the quotient  $(\PP^2)^6 // PGL(3)$.   This is a degree two covering of $\PP^4$ branched along a $\Sigma_6$-invariant quartic hypersurface $\cI_4 \subset \PP^4$.
$\cI_4$ is known as the {\it Igusa quartic}, it  is the Satake compactification of the moduli space $\cA_2(2)$ of principally polarized abelian surfaces with a level two structure \cite{ig:tc1}, embedded in $\PP^4$ by fourth powers of theta-constants.  The involution that defines the covering is the Gale transform (also called association, for details see \cite{do:pstf}, \cite{eis:pop}), which is defined as follows. 

\begin{definition} (\cite{eis:pop}, Def. 1.1) Let $r,s \in \ZZ$. Set $\gamma=r+s+2$, and let $\Gamma \subset \PP^r,$ $\Gamma' \subset \PP^s$ be ordered nondegenerate 
sets of $\gamma$ points represented by $\gamma \times (r+1)$ and $\gamma \times (s+1)$ matrices $G$ and $G'$, respectively. We say that $\Gamma'$ is the Gale transform of $\Gamma$ if there exists a nonsingular diagonal $\gamma \times \gamma$ matrix $D$ s.t. $G^T\cdot D \cdot G'=0$, where $G^T$ is the transposed matrix of $G$.

\end{definition}

The Gale transform acts trivially on the sets of 6 points in $\PP^2$ that lie on a smooth conic. The branch locus of the double covering is then, roughly speaking, the moduli space of 6 points on a conic and henceforth a birational model of the moduli space of 6 points on a line. One can say even more, in fact the GIT compactification of the moduli space of 6 points on a line is a cubic 3-fold in $\PP^4$, called the Segre cubic, and its projectively dual variety is the Igusa quartic (see \cite{koike}, \cite{hu:gsq} for details). From the projective geometry point of view the singular locus of $\cI_4$ is an abstract configuration of lines and points that make up a $15_3$ \textit{configuration}. This means the following: there are 15 distinguished lines and 15 distinguished points. Each line contains 3 of the points and through each point pass 3 lines (see \cite{dolgabstract} Sect. 9 for more). Moreover $\cI_4$ is the only hypersurface with such a singular locus in the pencil of $\Sigma_6$-invariant quartics in $\PP^4$ (\cite{huntnice}, Example 7).

\medskip


Now let us recall some results from the literature about $\cS\cU(3,\cO_C)$.
Of course the twist by $K_C$ is an isomorphism and it is easy to understand the analogous result for $\cS\cU_C(3,3K_C)$. Since $\PP^8$ is smooth, the image of the singular locus of $\cS\cU(3,\cO_C)$ and the singular locus of the branch locus coincide, i.e. $Sing(\cC_6)=\theta(Sing(\cS\cU(3)))$. On $\pic^1(C)$ we have the involution $\lambda: L \mapsto K_C \otimes L^{-1}$ that leaves $\Theta$ invariant.
%
%
Hence $\lambda$ induces an action on all the powers of $\Theta$ and in particular, on $|3\Theta|$. The linear system $|3\Theta|$ decomposes in two eigenspaces, respectively 4 and 3 dimensional. We call $\PP^4_e$ the 4-dimensional eigenspace. It turns out that it  cuts out on $\cC_6$ a reducible variety given by a double $\PP^3$ (which is indeed contained in $Sing(\cC_6)$) and a quartic hypersurface $I \subset \PP^4$.  After the twist by $K_C$,  the first component is precisely  $L_e$, whereas the quartic 3-fold is an Igusa quartic. 

\begin{lemma}
The intersection of the closure of the general fiber of $\pi_e$ with $Sing(\cC_6)$ is a $15_3$ configuration of lines and points.
\end{lemma}

\begin{proof}
Recall that $Sing(\cC_6)$ is the locus of theta divisors corresponding to decomposable bundles, 
we will prove the claim by constructing explicitly these bundles.
Let us denote by $\Delta_{3K_C}$ the closed subset of $|3K_C|$ given by the intersection with the big diagonal of $C^{(6)}$. 
Let us take $G=q_1+ \dots +q_6\in |3K_C| - \Delta_{3K_C}$, and let us consider the fiber of $\Phi_{K_C}$ over $G$. In order to guarantee the semi-stability of the vector bundles, the only totally decomposable bundles in the fiber of $G$ are all the 15 obtained by permuting the $q_i$'s in $\cO_C(q_1 + q_2)\oplus \cO_C(q_3+q_4)\oplus \cO_C(q_5+q_6)$. Let us now consider the bundles that decompose as the direct sum of a line bundle $L$ and a rank two indecomposable bundle. By the previous argument of semi-stability then $L$ must be of the type $\cO_C(q_i+q_j)$ for some $i,j \in \{1,\dots,6\}$ and $E$ should have fundamental divisor $\DD_E=\sum_{k\neq i,j} q_k$. Call $F$ the line bundle $\cO_C(\sum_{k\neq i,j} q_k) \equiv 3K_C -q_i-q_j$. It is easy to see that $\cS\cU_C(2, F)\cong \cS\cU_C(2,\cO_C)\cong \PP^3$, the isomorphism being given by the tensor product by a square root $F'$ of $F$. Now we recall from \cite{bol:kumwed} the following description of the fundamental map $\Phi_{F'}:\cS\cU_C(2, F) \dashrightarrow |F|$. The linear system $|F|$ is a $\PP^2$ and the fibers of $\Phi_{F'}$ are just lines  passing through $D\in |F|$ and the origin $[\cO_C\oplus\cO_C]$. Now the composition of the following embedding

\begin{eqnarray}
\zeta: \cS\cU_C(2, F) & \hookrightarrow & \cS\cU_C(3,3K),\\
\left[ E \right] & \mapsto & \left[ \cO_C(q_i+q_j) \oplus E \right],
\end{eqnarray}
 
with the theta map is linear. In fact $\zeta(\cS\cU_C(2, F))$ is contained in the branch locus and the associated $3\Theta$ divisors form a three dimensional linear subsystem isomorphic to $|2\Theta|$. 
Then the image of $\zeta$ intersects the closure $\overline{\Phi^{-1}_{K_C}(G)}$ of the fiber over $G$ exactly along the fiber of $\Phi_{F'}$ over the divisor $\sum_{k\neq i,j} q_k \in |F|$. By \cite{bol:kumwed} we know that this is a line and it is not difficult to see that it contains 3 of the 15 totally decomposable bundles. On the other hand each totally decomposable bundle with fundamental divisor $G$ is contained in three lines of this kind.

\end{proof}

\begin{remark}\label{15dege}
When the divisor $G$ is taken in $\Delta_{3K_C}$ then the $15_3$ configuration  degenerates because some of the points and some of the lines coincide.
\end{remark}

\begin{theorem}\label{fibrusa}
The closure of the general fiber of ${\Phi}_{K_C}$ is the GIT quotient \\ $(\PP^2)^6//PGL(3)$.
\end{theorem}

\begin{proof}
We recall that $L_e$ is contained in $Sing(\cC_6)$ and in particular \newline scheme-theoretically it is contained twice in $\cC_6$. Since ${\Phi}_{K_C} $ factors through the projection with center $L_e$, then $\overline{\Phi_{K_C}^{-1}(G)}$, for $G \in |3K_C| - \Delta_{3K_C}$, is a degree two cover of $\PP^4_G:= \overline{\pi_e^{-1}(G)}$ ramified along the intersection of $\cC_6$ with $\PP^4_G$ which is residual to $2L_e$. This intersection is then a quartic hypersurface in $\PP^4_G$. Notably, since $L_e \subset \PP_e^4$, there exist a point $T \in |3K_C|$ s.t. $\overline{\pi_e^{-1}(T)} \cap \cC_6 $ is an Igusa quartic (see Prop. 5.2 of \cite{oxpa:prv} or \cite{minhigusa} Sect. 4).


Let us now blow up $|3\Theta|$ along $L_e$ and call the resulting variety $\widetilde{\PP}^8$. This contains canonically the blown up Coble sextic, which we denote by $\widetilde{C}_6$. Then the rational map $\pi_e$ resolves in a proper, flat map $\widetilde{\pi}_e$ as in the following diagram.

\qquad \qquad \qquad \qquad \qquad $\xymatrix{ \widetilde{\PP}^8 \ar[d] \ar[dr]^{\widetilde{\pi}_e} &  \\
|3\Theta| \ar@{-->}[r]^{{\pi}_e}& |3K_C| }$

This says that the restriction of $\widetilde{\pi}_e$ to $\widetilde{C}_6$ is a flat family of quartic 3-folds over $|3K_C|$ and for any $B\in |3K_C|$ we have an isomorphism $\overline{\pi^{-1}_e(B)}\cap \cC_6 \cong \widetilde{\pi}^{-1}_e(B)\cap \tilde{C}_6$. Hence also one fiber of $\widetilde{\pi}_{e|\widetilde{C}_6}$ is an Igusa quartic. Since the ideal of the singular locus of $\cI_4$ is generated by the four polar cubics (\cite{hu:gsq}, Lemma 3.3.13) then the Igusa quartic has no infinitesimal deformations, i.e. it is rigid. This implies that the generic member of the flat family of quartics over $|3K_C|$ is an Igusa quartic. This in turn implies that the closure of the generic fiber of $\Phi_{K_C}$ is $(\PP^2)^6//PGL(3)$.
\end{proof}




\begin{corollary}
The Coble sextic $\cC_6$ is birational to a fibration over $\PP^4$ whose fibers are Igusa quartics.
\end{corollary}

\begin{corollary}
Along the generic fiber of $\Phi_{K_C}$, the involution of the degree two covering $\cS\cU(3,\cO_C)\ra |3\Theta|$ coincides with the involution given by the association isomorphism on $(\PP^2)^6//PGL(3)$.
\end{corollary}


\medskip

The fact that the intersection of $Sing(\cC_6)$ with the fibers over the open set $|3K_C| - \Delta_{3K_C}$ is precisely a $15_3$ configuration makes us argue that $|3K_C| - \Delta_{3K_C}$ should be the open locus where by rigidity the family of quartic three-folds is isotrivially isomorphic to the Igusa quartic. As already seen in Remark \ref{15dege}, if $B$ is an effective divisor out of this locus, then $\overline{\pi_e^{-1}(B)}\cap Sing(\cC_6)$ is a \textit{degenerate} $15_3$ configuration, in the sense that some of the 15 points and lines coincide. We are not able to prove the following, but it is tempting to say that this is all the singular locus of the special quartic three-folds over $\Delta_{3K}$. These would give very interesting examples of \textit{degenerate Igusa quartics}.
It would be interesting to study projective properties of these fibers such as the relation with the Segre cubics or with the Mumford-Knudsen compactification $\overline{\cM}_{0,6}$ of the moduli space of 6 points on a line. For instance, do they come from linear systems on $\overline{\cM}_{0,6}$? If it is so, what linear systems on $\overline{\cM}_{0,6}$ do they come from?


\medskip

The rational dual map of the Coble sextic has been thoroughly studied and described in \cite{orte:cob} and \cite{Quang}. Let us denote by $X_0,\dots ,X_8$ the coordinates on $\PP^8\cong |3\Theta|$ and by $F(X_0:\dots:X_8)$ the degree six poynomial defining $\cC_6$. Then the dual map is defined as follows:

\begin{eqnarray*}
\cD_6:\cC_6 & \dashrightarrow & \cC_3;\\
x & \mapsto & \left[\frac{\partial F}{\partial X_0}: \dots : \frac{\partial F}{\partial X_8}\right].
\end{eqnarray*}

The polar linear system is given by quintics that vanish along $Sing(\cC_6)$. Now fix a general divisor $B\in |3K|$ and let $\cI_B$ be the Igusa quartic defined by $(\cC_6\cap \overline{\pi_e^{-1}(B)}) - 2L_e$. Let us consider the restriction of $\cD_6$ to $\cI_B\subset \overline{{\pi}_e^{-1}(B)} =: \PP^4_B$ and denote by $A$ the $15_3$ configuration of points and lines in  $\PP^4_B$. Now let $H$ be the class of $L_e$ in $\pic(\PP^4_B)$ and consider the $4$- dimensional linear system $|\cI_A(3)+2H|$ on $\cI_B$. We can show the following.

\begin{proposition}\label{igudual}
The restricted dual map $\cD_{6|\cI_B}$ is given by a linear system $|\cD_{\cI_B}|$ that contains $|\cI_A(3)+2H|$ as a linear subsystem.
\end{proposition}


Remark that this means that for the general fiber $\cI_B$, there exists a canonical way to project the image $\cD(\cI_B)\subset \cC_3$ to a $\PP^4_B$ where the image of $\cI_B$ is a Segre cubic. This is summarized in the following.

\medskip

\begin{corollary}
The Coble cubic is birational to a fibration in Segre cubics over  $\PP^4$.
\end{corollary}

\begin{remark}
The birationality in itself is trivial, since $\cC_3$ is birational to $\cC_6$ which is birational to a fibration in Igusa quartics (which are in turn all birational to the Segre cubic) over $\PP^4$. The projections on the linear systems $|\cI_A(3)+2H|$ give a constructive canonical way to realize it. 
\end{remark}



\subsection{The Coble quartic}

In this subsection we assume that $C$ is a curve of genus $3$ and we consider the  moduli space $\cS\cU(2,\cO_C)$.  We recall that  the Kummer variety $Kum(C):=J(C)/\pm Id$ of $C$ is contained naturally in the $2\Theta$-linear series, whereas the moduli space $\cS\cU(2,\cO_C)$ is embedded by $\theta$ in $\PP^7\cong |2\Theta|$ as the unique quartic hypersurface $\cC_4$ singular along $Kum(C)$. This hypersurface is called the \textit{Coble quartic}. It is also known \cite{paulydual} that the Coble quartic is projectively self dual.
\hfill\par
Now we  recall some properties of the \textit{Segre cubic}. This is a nodal (and hence rational) cubic three-fold $S_3$ in $\PP^4$ whose singular locus is given by ten double points. There is a natural action of $\Sigma_6$ on this projective space and $S_3$ is invariant with respect to this action. $S_3$ is in fact the GIT quotient $(\PP^1)^6//PGL(2)$ \cite{do:pstf}. Moreover, $S_3$ realizes the so-called \textit{Varchenko bound}, that is, it has the maximum number of double points (ten) that a cubic threefold with isolated singularities may have and  this property identifies the Segre cubic in a unique way, up to projective equivalence. As already stated it is the projective dual variety of the Igusa quartic.

Our construction allows us to give a simple proof of the following result from \cite{albol}.

\begin{proposition}\label{quartic}
The moduli space $\cS\cU(2,\cO_C)$ is birational to a fibration over  $\PP^3$ whose fibers are Segre cubics.
\end{proposition}

\begin{proof}
First of all we twist all vector bundles by a degree 3 divisor $D$, thus obtaining the isomorphic moduli space $\cS\cU(2,\cO_C(2D))$. Let ${\Phi}_D \colon  \cS\cU(2,\cO_C(2D)) \to \vert \cO_C(2D) \vert\cong \PP^3$ be the composition of $b_D$ and the fundamental map $\Phi$. By Thm. \ref{thetabirat}, since $\theta_D$ is an embedding, we can identify  $\cS\cU(2,\cO_C(2D))$ with its image $\cC_4\subset \PP^7$ and  ${\Phi}_D$ with the linear projection  onto  $|2D|$. The center of the projection is the linear span of the locus of vector bundles $E$ s.t. $h^0(C,E\otimes\cO_C(D))>2$. Suppose $E$ is stable. Since it has rank 2 and trivial determinant, then $E\cong E^*$ and by an easy Riemann-Roch computation we find that $h^0(C,E\otimes\cO_C(D))>2$ if and only if $h^0(C,E\otimes \cO_C(K-D)) > 0$. This is equivalent to the fact the $E$ lies in the $\PP^3\cong|3K-2D|^*\subset \cC_4$ that parametrizes vector bundles $E$ that can be written as an extension of the following type

$$0\ra \cO_C(D-K) \ra E \ra \cO_C(K-D)\ra 0.$$

Let us denote by $\PP^3_c$ this projective space. $\cC_4$ contains $\PP^3_c$ with multiplicity one. This implies that the closure of any fiber of the projection ${\Phi}_D :\cC_4 \dashrightarrow \vert 2 D \vert$ is a cubic 3-fold contained in the $\PP^4$ spanned by $\PP^3_c$ and a point of $\vert 2 D \vert $. Let us denote as usual by $\Delta_D$ the intersection of the large diagonal with the linear system $|2D|\subset C^{(6)}$. Then suppose we fix a $B\in |2D| - \Delta_D$. Let us consider the intersection of the fiber of $\Phi_D$ over $B$ with the strictly semi-stable locus. By semi-stability it is easy to see that these points correspond to the partitions of the 6 points of $B$ in complementary subsets of 3 elements each. We have ten of them. As stated above, a cubic 3-fold cannot have more than ten ordinary double points and the Segre cubic is uniquely defined by this singular locus up to projective equivalence.
\end{proof}


Also in this case, if $B\in \Delta_D$ then the intersection $\overline{\Phi_D^{-1}(B)}\cap Kum(C)$ is set-theoretically a finite set of points of cardinality strictly smaller then 10: the singular locus seems to degenerate. It is tempting, as in the case of Igusa quartics, to say that some of these points have multiplicity bigger than one and we obtain \textit{degenerate Segre cubics} over $\Delta_D$.

\medskip

As we have already remarked, in the case of $\cC_4$ the polar map is also well known and described. Let $Y_i$ be the coordinates on $\PP^7\cong |2\Theta|$ and $G(Y_0:\dots:Y_7)$ the quartic equation defining $\cC_4$, then the (self) polar rational map of $\cC_4$ is defined in the following way.

\begin{eqnarray*}
\cD_4:\cC_4 & \dashrightarrow & \cC_4;\\
x & \mapsto & \left[\frac{\partial G}{\partial Y_0}: \dots : \frac{\partial G}{\partial Y_7}\right].
\end{eqnarray*}

Let $B\in |2D| - \Delta_D$ and let $\PP^4_B$ be the linear span of the point corresponding to $B$ and of $\PP^3_c$
It turns out that the restriction of $\cD_4$ to $\PP^4_B$ beahaves in a way very similar to the case of the sextic (see Prop. \ref{igudual}).. Let $S_{3B}\subset \PP^4_B$ be the Segre cubic such that $\cC_4 \cap \PP^4_B = S_{3B} \cup \PP^3_c$. We denote by $J$ the set of 10 nodes of $S_{3B}$. Then the linear series $|\cI_J(2)|$ on $S_{3B}$ is the polar system of the Segre cubic.

\begin{proposition}\label{segdual}
The restricted dual map $\cD_{4|S_{3B}}$ is given by a linear system $|\cD_{S_3}|$ that contains $|\cI_J(2)+H|$ as a linear subsystem.
\end{proposition}

As in the case of $\cC_6$ this implies that we have a canonical way to construct the birational map of the following corollary via the polar map $\cD_4$.

\begin{corollary}
The Coble quartic is birational to a fibration in Igusa quartics over $\PP^3$.
\end{corollary}




\begin{thebibliography}{30}

\bibitem{albol}
Alberto Alzati, and Michele Bolognesi,
{\em A structure theorem for $\mathcal{SU}_C(2)$ and the moduli of pointed genus zero curves},
Preprint, hyyp://arxiv.org/0903.5515, 1--20.

\bibitem{ang}
Christian Anghel, {\em Fibrés vectoriels semi-stables sur une courbe de genre deux et association des points dans l'espace projectif},
Serdica Math J. {\bf 30} (2004), no.2-3, 103--110.

\bibitem{acgh:gac}
E.Arbarello, M.Cornalba,P.A. Griffiths, and J.Harris, {\em Geometry of algebraic curves. Vol. I}, Grundlehren der Mathematischen Wissenschaften, vol. 267, Springer-Verlag, New York, 1985.


\bibitem{bove23}
A. Beauville,
{\em Vector bundles and theta functions on curves of genus 2 and 3},
Amer. J. Math {\bf 128} (2006), no.3, 607--618.

\bibitem{bnr}
A.Beauville, M.S.Narasimhan, and S.Ramanan,
{\em Spectral curves and the generalized theta divisor},
J. Reine Angew. Math {\bf 398} (1989), 169--179.

\bibitem{ab:rk2}
A.Bertram,
{\em Moduli of rank-2 vector bundles, theta divisors, and the geometry of curves in the projective space},
J. Differential Geom. {\bf 35} (1992), no.2, 429-469.


\bibitem{bol:kumwed}
M. Bolognesi,
{\em A conic bundle degenerationg on the Kummer surface},
Math. Zeit. {\bf 261}, (2009), no.1, 149--168.


\bibitem{BGPMN}
S.B. Bradlow, O. Garcia-Prada, V. Mu{\~n}oz., and P.E. Newstead
{\em Coherent systems and Brill-Noether theory},
Internat. J. Math., {\bf 14} (2003), no.7, 683--733.

\bibitem{hidalgos}
S.B. Bradlow, O. Garcia-Prada, V. Mercat, V. Mu{\~n}oz., and P.E. Newstead
{\em Moduli spaces of coherent systems of small slope on algebraic curves},
Commun. in Algebra, Vol. {\bf 37}, no. 8, pp. 2649-2678, (2009).

\bibitem{BVtheta}
S.Brivio, and A.Verra,
{\em The theta divisor of $SU_C(2,2d)^s$ is very ample if $C$ is not hyperelliptic},
Duke Math. J. {\bf 82} (1996), 503-552.


\bibitem{briverraBN}
S.Brivio, and A.Verra,
{\em The Brill-Noether curve of a stable vector bundle on a genus two curve}
Algebraic cycles and motives, Vol.2, London Math. Soc. Lecture Note Ser., vol.344, Cambridge Univ. Press, Cambridge, 2007, 73--93.


\bibitem{briverraPL}
S.Brivio, and A.Verra,
{\em Plucker forms and the theta map},
Preprint http://arxiv.org/abs/0910.5630, to appear on American J. of Math., 1--20.


\bibitem{coble}
A.B.Coble,
{\em Algebraic Geometry and theta functions},
American Math. Society Colloquium Pubblications, vol. 10, AMS, Providence, R.I., 1982, Reprint of the 1929 edition.



\bibitem{dolgabstract} 
I.Dolgachev,
{\em Abstract configurations in algebraic geometry},
The Fano conference, Univ. Torino, Turin, 2004, 423--462.


\bibitem{do:pstf}
I. Dolgachev, and D.Ortland,
{\em Point sets in the projective spaces and theta functions},
Asterisque, vol.165, Soc. Math. France, 1988.

\bibitem{dn:pfv}
J-M. Drezet, and M.S. Narasimhan,
{\em Groupe de Picard des variétés de modules de fibrés semi-stables sur les courbes algébriques},
Invent. Math. {\bf 97} (1989), no.1, 53--94.

\bibitem{eis:pop}
D.Eisenbud, and S.Popescu,
{\em The projective geometry of the Gale transform},
J.Algebra, {\bf 230}, (2000), no.1, 127--173.


\bibitem{huntnice}
B.Hunt,
{\em Nice modular varieties},
Experiment. Math. {\bf 9}, (2000), no.4, 613--622.



\bibitem{hu:gsq}
B.Hunt,
{\em The geometry of some special arithmetic quotients},
LNM, vol. 1637, Springer, Berlin, 1996.


\bibitem{ig:tc1}
J.-I.Igusa,
{\em On the graded ring of theta-constants.},
Amer. J. Math. {\bf 86} (1964), 219--246.

\bibitem{vgiz}
E.Izadi, and L. Van Geemen
{\em The tangent space to the moduli space of vector bundles on a curve and the singular locus of the theta divisor of the Jacobian},
J. Alg. Geom. {\bf 10} (2001), no.1, 133-177.


\bibitem{kingnew}
 A. D. King, and P.E. Newstead,
{\em Moduli of Brill-Noether pairs on algebraic curves},
Internat. J. Math. {\bf 6} (1995), no. 5, 733--748.


\bibitem{daking}
A.King, and A.Schofield,
{\em Rationality of moduli of vector bundles on curves},
Indag. Math. (N.S.) {\bf 10} (1999), no.4, 519--535.


\bibitem{koike}
K.Koike,
{\em Remarks on the Segre cubic.},
Arch. Math. (Basel) {\bf 81} (2003), no.2, 155--160.

\bibitem{ivolocale}
Y.Laszlo,
{\em Local structure of the moduli space of vector bundles over curves}
Comment. Math. Helv {\bf 71} (1996), no.3, 373--401.


\bibitem{MumNews}
D.Mumford, and P.Newstead,
{\em Periods of a moduli space of vector bundles over curves}
Amer. J. Math. {\bf 90} (1968), 1200-1208.

\bibitem{ramanara}
M.S. Narasimhan, and S.Ramanan,
{\em Moduli of vector bundles on a compact Riemann surface}
Ann. of Math (2) {\bf 89} (1969), 14--51.



\bibitem{rana:cra}
M.S. Narasimhan, and S.Ramanan,
{\em Vector bundles on curves},
Algebraic Geometry (internat. Colloq., Tata Inst. Fund. Res., Bombay, 1968), Oxford Univ. Press, London, 1969, 335-346.



\bibitem{Quang}
Q.M.Nguyen,
{\em Vector bundles, dualities and classical geometry on a curve of genus two},
Internat. J. Math {\bf 18} (2007), no.5, 535--558.


\bibitem{minhigusa}
Q.M.Nguyen, and S.Rams,
{\em On the geometry of the Coble-Dolgachev sextic}
Le Matematiche (Catania) {\bf 58} (2003), no.2, 257--275 (2005).

\bibitem{orte:cob}
A.Ortega,
{\em On the moduli space of rank 3 vector bundles on a genus 2 curve and the Coble cubic},
J. Alg. Geom. {\bf 14} (2005), no.2, 327--356.


\bibitem{oxpa:prv}
W.Oxbury, and C.Pauly,
{\em $SU_C(2)$-Verlinde spaces as theta spaces on Pryms.},
Internat. J. Math. {\bf 7} (1996), no.3, 393-410.

\bibitem{oxpau:heis}
W.Oxbury, and C.Pauly,
{\em Heisenberg invariant quartics and $SU_C(2)$ for a curve of genus four},
Math. Proc. Cambridge Philos. Soc. {\bf 125} (1999), no.2, 295--319.

\bibitem{paulydual}
C.Pauly,
{\em Self-duality of Coble's quartic hypersurface and applications},
Mich. Math. Journ. {\bf 50} (2002), no.3, 551-574.


\bibitem{deg16}
C.Pauly,
{\em Rank four vector bundles without theta divisor over a curve of genus two},
To appear on Adv. in Geom., 1--8, http://arxiv.org/abs/0804.3001



\bibitem{raysect}
M.Raynaud,
{\em Sections des fibrés vectoriels sur une courbe},
Bull. SMF {\bf 110} (1982), no.1, 103--125.




\bibitem{stacchi}
A.Vistoli,
{\em Grothendieck topologies, fibered categories and descent theory},
Fundamental algebraic geometry, Math.Surveys Monogr., vol 123, Amer. Math. Soc., Providence, RI, 2005, 1--104.


\end{thebibliography}

 \end{document}